\documentclass{amsart}
\usepackage{amssymb,txfonts, pxfonts, amssymb,amsmath, latexsym, amsfonts,graphicx}
\usepackage{verbatim,amscd, mathabx}
\usepackage{amssymb,txfonts, pxfonts, amssymb,amsmath, latexsym, amsfonts}
\makeatletter
\def\l@section{\@tocline{1}{0pt}{1pc}{}{}}
\def\l@subsection{\@tocline{2}{0pt}{1pc}{4.6em}{}}
\def\l@subsubsection{\@tocline{3}{0pt}{1pc}{7.6em}{}}
\renewcommand{\tocsection}[3]{%
  \indentlabel{\@ifnotempty{#2}{\makebox[2.3em][l]{%
    \ignorespaces#1 #2.\hfill}}}#3}
\renewcommand{\tocsubsection}[3]{%
  \indentlabel{\@ifnotempty{#2}{\hspace*{2.3em}\makebox[2.3em][l]{%
    \ignorespaces#1 #2.\hfill}}}#3}
\renewcommand{\tocsubsubsection}[3]{%
  \indentlabel{\@ifnotempty{#2}{\hspace*{4.6em}\makebox[3em][l]{%
    \ignorespaces#1 #2.\hfill}}}#3}
\makeatother

\newtheorem{theorem}{Theorem}[section]

\newtheorem{proposition}[theorem]{Proposition}

\theoremstyle{definition}

\theoremstyle{remark}

\numberwithin{equation}{section}

\newcommand{\R}{\mathbb{R}}
\newcommand{\N}{\mathbb{N}}

\newcommand{\X}{\mathrm{X}}

\newcommand{\Y}{\mathrm{Y}}

\newcommand{\sign}{\mathrm{sign}}

\renewcommand{\S}{\mathbf{S}}

\newcommand{\essinf}{\mathop{\rm ess\ inf}}
\newcommand{\esssup}{\mathop{\rm ess\ sup}}

\newcommand{\vep}{\varepsilon}
\renewcommand{\epsilon}{\vep}
\newcommand{\eps}{\vep}

\begin{document}

\title{Duality of ODE-determined norms}

\author{Jarno Talponen}
\address{University of Eastern Finland\\Institute of Mathematics\\Box 111\\FI-80101 Joensuu\\Finland}
\email{talponen@iki.fi}

\date{\today}

\begin{abstract}
Recently the author initiated a novel approach to varying exponent Lebesgue space $L^{p(\cdot)}$ norms. In this approach the norm is defined by means of weak solutions to suitable first order ordinary differential equations (ODE). The resulting norm is equivalent with constant $2$ to a corresponding Nakano norm but the norms do not coincide in general and thus their isometric properties are different. In this paper the duality of these ODE-determined $L^{p(\cdot)}$ spaces is investigated. It turns out that the duality of the classical $L^p$ spaces generalizes nicely to this class of spaces. This duality is obtained 
in the isometric sense. The superreflexivity of these spaces is characterized under the anticipated conditions. A kind of universal space construction is also given for these spaces.
\end{abstract}

\maketitle

\section{Introduction}
The author recently introduced a means to construct and analyze varying exponent $L^{p(\cdot)} [0,1]$ norms by applying 
Carath\'eodory's weak solutions to suitable ordinary differential equations (ODE). Here we investigate the duality of such Banach function spaces. This leads to looking at the geometry and duality of Banach spaces in terms of the properties of the corresponding differential equations. 

The classical Birnbaum-Orlicz norms were defined in the 1930's, and since then there have been various generalizations of these norms to several directions. Notable examples of norms and spaces carry names such as Amemiya, Besov,  Lizorkin, Lorentz, Luxemburg, Musielak, Nakano,  Orlicz,  Triebel, Zygmund, see e.g. \cite{bo}, \cite{lux}, \cite{mus}. These norms have been recently applied to other areas 
of mathematics as well as to some real-world applications, see e.g. \cite{rr2}. Roughly speaking, these norms can be viewed as belonging to a family of derivatives of the Minkowski functional.
This kind of approach leads to several varying exponent $L^{p(\cdot)}$ type constructions, e.g. for sequence spaces, Lebesgue spaces, Hardy spaces and Sobolev spaces. There is a vast literature on these topics, see \cite{kovacik}, \cite{LT}, \cite{nakai} and \cite{rr} for samples and further references. There are also other ways of looking at the varying exponent $L^p$ spaces, such as the Marcinkiewicz space,
whose approach differs from the one mentioned above, see \cite{mar}.

Let us recall that the general Nakano or Musielak-Orlicz type norms are defined as follows:
\[\|f\|=\inf \left\{\lambda>0 \colon \int_\Omega \phi\left(\frac{|f(t)|}{\lambda},t\right)\ dm(t)\leq 1\right\}.\]
Here $\phi$ is a positive function satisfying suitable structural conditions. For instance, $\phi(s,t)=s^{p(t)}$, or 
$\psi(s,t)=\frac{s^{p(t)}}{p(t)}$, $1\leq p(\cdot) <\infty$, produces a norm that can be seen as a varying exponent $L^p$ norm. In the latter case we use the name \emph{Nakano norm} (cf. \cite{JKL,Maligranda}), which is of particular interest in this paper.

The construction of the $L^{p(\cdot)} [0,1]$ norms studied in this paper is fundamentally different from the above Minkowski functional derivatives and was introduced in \cite{talponen3} as a 'continuous version' of certain sequence spaces. These sequence spaces can be described as varying exponent $\ell^p$ spaces, or $\ell^{p(\cdot)}$ spaces, which first appear in \cite{Sobczyk} and were later studied in \cite{talponen}, cf. \cite{kalton1}. The construction of these spaces is rather natural and local in its nature. 

The accumulation of the norm is captured by a suitable ODE in such a way that its weak solution, $\varphi_f \colon [0,1]\to [0,\infty)$, shall represent the norm as follows:
\begin{equation}\label{eq: main}
\varphi_f (t)=\|1_{[0,t]}f\|,
\end{equation}
so that in particular $\varphi_ f (0)=0$ and $\varphi_ f (1)=\|f\|$. This absolutely continuous function 
obeys the following ODE:
\[\varphi_f (0)=0 ,\ \varphi_{f}'(t)=\frac{|f(t)|^{p(t)}}{p(t)}\varphi(t)^{1-p(t)}\quad \mathrm{for\ a.e.}\ t\in [0,1].\]
In the constant $p$ case the above ODE is a separable one, and solving it yields 
\[(\varphi_{f} (t))^p = \int_{0}^t |f(s)|^p \ ds\]
which clearly coincides with the classical definition of the $L^p$ norm. The ODE-determined $L^{p(\cdot)}$ class is 
\[L^{p(\cdot)} = \{f \in L^0 \colon \varphi_f \ \text{exists\ and}\ \varphi_f (1)<\infty \}\]
as a set where we identify functions which coincide almost everywhere. For an unbounded exponent $p(\cdot)$ it may happen that $L^{p(\cdot)}$ is not a linear space, but if it is, then the solutions define a norm:
\[\|f\|_{L^{p(\cdot)}} := \varphi_f (1).\]
This is equivalent with constant $2$ to the particular Nakano norm (see $\psi$ above).

Thus, we continue the analysis of the $L^{p(\cdot)}$ spaces in the sense of \cite{talponen3}. Hopefully, some clean findings obtained here justify the fact that the definition of these ODE-determined spaces is rather natural. For instance, it turns out that the duality and the superreflexivity of these function spaces 
behave exactly in the anticipated way. Here the norm satisfies H\"{o}lder's inequality properly, i.e. without 
any additional constant, and the spaces $L^{p(\cdot)}$ and $L^{p^* (\cdot)}$ become isometrically
dual to each other in case $1 < \mathrm{ess\ inf}\ p(\cdot) \leq \mathrm{ess\ sup}\ p(\cdot) <\infty$.

\subsection{Preliminaries and auxiliary results}

We will usually consider the unit interval $[0,1]$ endowed with the Lebesgue measure $m$. 
Here for almost every (a.e.) refers to $m$-a.e., unless otherwise specified. 
Denote by $L^0$ the space of Lebesgue-to-Borel measurable functions on the unit interval.
We denote by $\ell^0 (\N)$ the vector space of sequences of real numbers with point-wise operations.
The monographs \cite{DE_book}, \cite{FA_book} and \cite{LT} provide suitable general background information. The paper \cite{talponen3} provides the necessary prerequisite background information, including definitions, basic results and the heuristic motivation of the construction, cf. \cite{talponen2}. 

We will study Carath\'eodory's weak formulation to ODEs, that is, in the sense of Picard type integral formulation, where solutions are required to be only absolutely continuous. This means that, given an ODE 
\[\varphi(0)=x_0 ,\ \varphi'(t)=\Theta(\varphi(t),t),\quad \mathrm{for\ a.e.}\ t\in [0,1],\]
we call $\varphi$ a weak solution in the sense of Carath\'eodory if $\varphi$ is absolutely continuous, $t\mapsto  \Theta(\varphi(t),t)$ is measurable and 
\[\varphi(T)= x_0 + \int_{0}^T \Theta(\varphi(t),t)\ dt\]
holds for all $T\in [0,1]$, where the integral is the Lebesgue integral. In what follows, we will refer to Carath\'eodory's solutions simply as solutions.

Whenever we make a statement about a derivative we implicitly state that it exists. 
We will write $F\leq G$, involving elements of $L^0$, if $F(t)\leq G(t)$ for a.e. $t\in [0,1]$.
We denote the characteristic function or indicator function by $1_A$ defined by $1_A (x)=1$ if $x\in A$
and $1_A (x)=0$ otherwise.

We will frequently calculate terms of the form $(a^p + b^p)^{\frac{1}{p}}$ where $a,b\geq 0$ and 
$1\leq p <\infty$. We will adopt from \cite{talponen} the following shorthand notation for this:
\[a \boxplus_p b = (a^p + b^p)^{\frac{1}{p}}.\]
This defines a commutative semi-group on $\R_+$, in particular, the associativity
\[a \boxplus_p (b \boxplus_p c )=(a \boxplus_p b) \boxplus_p c ,\]
is useful. In taking a sequence of $\boxplus_p$ or $\oplus_p$ operations we always perform the operations
from left to right, unless there are parentheses indicating another order.
We will also use the following operation:
\[\bigboxplus_{1\leq i\leq n}^p x_i = x_1 \boxplus_p x_2 \boxplus_p \ldots \boxplus_p x_n  =\left(\sum_{i=1}^n x^p _i \right)^{\frac{1}{p}},\quad x_1 , \ldots ,x_n \in \R_+ .\]

The space $\ell^{p(\cdot)} \subset \ell^0$, $p\colon \N \to [1,\infty)$, consists of those elements $(x_n)$ such that the 
following limit of a non-decreasing sequence exists and is finite:
\[\lim_{n\to \infty} (\ldots (((|x_1 | \boxplus_{p(1)} |x_2 |)\boxplus_{p(2)} |x_3 |) \boxplus_{p(3)} |x_4 |)\boxplus_{p(4)}\ldots \boxplus_{p(n-1)}|x_n |)\boxplus_{p(n)}|x_{n+1}|\]
and the above limit becomes the norm of the space, see \cite{talponen}.

The author is grateful to Professors Pilar Cembranos and Jos\'{e} Mendoza for providing an argument of the following fact in a personal communication.

\begin{proposition} 
Let $1\leq p\leq r <\infty$ and $A_k = (a_{ij}^{(k)}) \in \ell^r (\ell^p )$, $k \in \N$, with 
non-negative entries and $a_{ij}^{(k)}a_{ij}^{(l)}=0$ for all $i,j,k,l\in \N$, $k\neq l$.
Then 
\[\left\|\sum_{k\in \N} A_k \right\|_{\ell^r (\ell^p )} \leq \bigboxplus_{k\in\N} ^p \|A_k\|_{\ell^r (\ell^p )} .\]
\end{proposition}
\begin{proof}
With the above assumptions we have
\begin{multline*}
\left\|\sum_{k=1}^\infty A_k \right\|_{\ell^r (\ell^p )} ^p =\left(\sum_{i=1}^\infty \left(\sum_{j=1}^\infty \sum_{k=1}^\infty |a_{ij}^{(k)}|^p \right)^{\frac{r}{p}}\right)^{\frac{p}{r}}
=  \left\|\sum_{k=1}^\infty \left(\sum_{j=1}^\infty |a_{ij}^{(k)}|^p \right)_{i=1}^\infty \right\|_{\ell^{\frac{r}{p}}}\\
\leq \sum_{k=1}^\infty \left\|\left(\sum_{j=1}^\infty |a_{ij}^{(k)}|^p \right)_{i=1}^\infty \right\|_{\ell^{\frac{r}{p}}} = \sum_{k=1}^\infty \left(\sum_{i=1}^\infty \left(\sum_{j=1}^\infty |a_{ij}^{(k)}|^p \right)^\frac{r}{p} \right)^\frac{p}{r}=\sum_{k=1}^\infty \|A_k\|_{\ell^r (\ell^p )}^p .
\end{multline*}
\end{proof}

This in turn implies the following fact by decomposing the matrix to columns.
        
\begin{proposition}\label{prop: boxplus}
If $1\leq p\leq r <\infty$ and $(x_{ij})$ is an infinite matrix of non-negative numbers, then 
\[\bigboxplus^r _{j\in\N} \bigboxplus^p _{i\in\N} x_{ij} \leq  \bigboxplus^p _{i\in\N} \bigboxplus^r _{j\in\N} x_{ij}.\]
Equivalently, taking the transpose $T \colon (x_{ij}) \mapsto (x_{ji})$ defines a norm-$1$ operator $\ell^p (\ell^r )\to \ell^r (\ell^p )$.
\end{proposition}
\qed

The inequality in Proposition \ref{prop: boxplus} can be seen as a `distributive version' of the following fact appearing in \cite{talponen}:
\[a \boxplus_r (b \boxplus_p c ) \leq (a \boxplus_r b) \boxplus_p c ,\quad 1\leq p\leq r\leq \infty,\ a,b,c\in \R_+ .\]

In the context of function spaces an \emph{exponent} is a function $p \in L^0 [0,1]$ with $p\geq 1$.

Consider the following identity
\begin{equation}\label{eq: phiDelta}
\begin{split}
\varphi(t_0 +\Delta)&=(\varphi(t_0 )^{p(t_0 )}+\Delta |f(t_0 )|^{p(t_0 )})^{1/p(t_0 )},\\
&=\varphi(t_0 ) \boxplus_{p(t_0 )} \Delta^{1/p(t_0 )}|f(t_0 )|
\end{split}
\end{equation}
analogous to the $\ell^{p(\cdot)}$ construction. By taking the right derivative of \eqref{eq: phiDelta} we find a natural candidate for the norm-determining differential equation:
\begin{equation}\label{eq: diff1}
\frac{\partial^+}{\partial\Delta}\varphi(t_0 +\Delta)\bigg|_{\Delta=0}=\frac{|f(t_0 )|^{p(t_0 )}}{p(t_0 )}\varphi(t_0 )^{1-p(t_0 )}.
\end{equation}
Here we set $\Delta=0$, because we are interested in (infinitesimal) increments around $t_0$. 
So, the above equation is right if $f$ and $p$ are constant on the interval $[t_0 ,t_0 +\Delta]$, but the equation does \emph{not} concern the values of $f$, $\varphi$ and $p$ \emph{beyond} $t_0$.

In formulating the differential equation we do not require $f$ or $p$ to be continuous anywhere, but motivated by Lusin's theorem and related considerations we will use the above formula in any case and aim to define
$\varphi$ by 
\begin{equation} 
\varphi(0)=0 ,\ \varphi'(t)=\frac{|f(t)|^{p(t)}}{p(t)}\varphi(t)^{1-p(t)}\quad \mathrm{for\ a.e.}\ t\in [0,1].
\end{equation}
This formulation has the drawback that $0^{1-p(t)}$ is not defined. Also, it has a trivial solution $\varphi \equiv 0$, regardless of the values of $f$ if we use the convention $0^0 =0$ and $p\equiv 1$.
The behavior of the solutions is difficult to deal with in the case where $\varphi(t)$ is small and $p(t)$ is large.

To fix these issues, we will consider \emph{stabilized} solutions to the above initial value problem. Namely, we will use initial values $\varphi(0)={x_0}>0$, and to correct the error incurred we let ${x_0}\searrow 0$. The unique solutions $\varphi_{{x_0}}$ decreasingly converge point-wise to $\varphi$ which again satisfies the same ODE (where applicable). So, this procedure yields a unique maximal solution $\varphi$ which we will formulate, by slight abuse of notation, as
\begin{equation}\label{eq: def}
\varphi(0 )=0^+ ,\ \varphi'(t)=\frac{|f(t)|^{p(t)}}{p(t)}\varphi(t)^{1-p(t)}\quad \mathrm{for\ a.e.}\ t\in [0,1].
\end{equation}
There is more to the above procedure than merely picking a maximal solution; it turns out that in many situations it is convenient to look at positive-initial-value solutions first.
      
The above ODE is a separable one for a constant $p(\cdot)\equiv p$, $1\leq p<\infty$, and solving it yields 
\[(\varphi_f (1))^{p}=\int_{0}^{1}|f(t)|^p \ dt,\] 
compatible with the classical definition of the $L^p$ norm.
If $p(\cdot)$ is locally bounded and $|f(t)|^{p(t)}$ is locally integrable, then Picard iteration performed locally yields a unique solution for each initial value $\varphi(0)=a>0$.

We define the varying exponent space $L^{p(\cdot)}\subset L^0$ as the space of those functions $f\in L^0$ such that $\varphi_f (1)<\infty$ where $\varphi_f$ exists as an absolutely continuous solution to \eqref{eq: def} and the norm of $f$ will be $\varphi_f (1)$, see \cite{talponen3}. As usual, we use point-wise linear operations defined almost everywhere and we identify functions which coincide almost everywhere. As observed in \cite{talponen3}, a class $L^{p(\cdot)}$ need not always be a linear space. However, if 
$\esssup p < \infty$, for instance, then $L^{p(\cdot)}$ is a Banach space. In such a case the norm 
is equivalent to the Nakano norm appearing in the introduction, although these norms do not coincide in general. 

The issues with the linearity can be circumvented by extending the class. 
We consider functions $f\in L^0$ and define
\[N(f) :=\sup_{n\in\N} \|1_{p(\cdot)\leq n}\ f\|_{L^{p(\cdot)}}.\]
The extended function space $\widetilde{L}^{p(\cdot)}$ is the class of all functions $f$ with 
$N(f)<\infty$, and then $N$ becomes a norm for this space and $\widetilde{L}^{p(\cdot)}$ is in fact 
a Banach space. Let us take a subclass $L^{p(\cdot)}_0 \subset L^{p(\cdot)}$ defined as a closed subspace
\[L^{p(\cdot)}_0 := \overline{\left\{1_{p(\cdot )\leq n}\ f\colon f\in L^{p(\cdot)},\ n\in \N\right\}}\subset 
\widetilde{L}^{p(\cdot)}.\]
Then $L^{p(\cdot)}_0$ becomes a Banach space with the ODE-determined norm. 
The following fact, which can be obtained easily from the essentially bounded exponent case (see the results in \cite{talponen3}), gathers many cases where $L^{p(\cdot)}$ itself is in fact linear. 
\begin{theorem}
Let $p\in L^0 [0,1]$, $p(\cdot) \geq 1$. Assume that there is a family $\Gamma$ of mutually disjoint open intervals 
$I \subset (0,1)$ such that the following conditions hold:
\begin{enumerate}
\item For each $I \in \Gamma$ the exponent $p(\cdot)$ is essentially bounded on $I$;
\item $m\left(\bigcup \Gamma \right)=1$, 
\item For each $\Gamma_0 \subset \Gamma$ with $s=\sup \bigcup \Gamma_0 <1$, or $s=0$, there is 
$I \in \Gamma$ such that $\inf I =s$. That is, $\Gamma$ is `well-ordered'. 
\end{enumerate}
Then $L^{p(\cdot)} = \widetilde{L}^{p(\cdot)}$ and is in particular a Banach space. 
\end{theorem}
\begin{proof}
It is clear that $L^{p(\cdot)} \subset \widetilde{L}^{p(\cdot)}$. Therefore we are required to verify that for each 
$f\in \widetilde{L}^{p(\cdot)}$ there is a properly defined solution $\varphi_f$ such that 
\[\varphi_f (t)= \|1_{[0,t]} f\|_{\widetilde{L}^{p(\cdot)}}.\]
Clearly the above solution is defined in the first interval $I_0 \in \Gamma$, since the exponent is essentially bounded there 
(see \cite{talponen3}). Suppose that $\varphi_f $ has been properly defined on $[0,s )$, $0<s\leq 1$. If $s=1$ then 
there is nothing to prove. In case $s<1$ there is according to 
the assumptions on $\Gamma$ an interval $I \in \Gamma$ such that $\inf I \leq s < \sup I$. Then, again by virtue of the essentially bounded exponent on $I$, we may 
further extend the solution to $[0, \sup I ]$. Inductively, it follows that $\varphi_f $ can be properly defined on the whole 
unit interval.      
\end{proof}

Note that there is a more direct way to obtain H\"{o}lder's inequality, than the argument provided in \cite{talponen3}.
Namely, let $p(\cdot)>1$ be any measurable exponent and define the conjugate exponent $p^*$ by 
$\frac{1}{p(t)}+ \frac{1}{p^* (t)}=1$, let $f \in L^{p(\cdot)}$ and $g\in L^{p^* (\cdot)}$.
Consider any point $t_0 \in (0,1)$ such that $\int_{0}^{t_0} |fg|\ ds>0$, thus  
$\varphi_f (t_0 ),  \varphi_g (t_0 ) >0$. Then, using homogeneity, we may assume with out loss of generality that 
$ \varphi_f (t_0 ) =  \varphi_g (t_0 ) =1$ above. It follows by Young's inequality and the definition and normalizations 
of the solutions that 
\begin{multline*}
\frac{d}{dt} \int_{0}^{t} |fg|\ ds\ \bigg|_{t=t_0 } = |f(t_0 )g(t_0 )| \leq \frac{|f(t_0 )|^{p(t_0 )}}{p(t_0 )} 
+\frac{|g(t_0 )|^{p^* (t_0 )}}{p^* (t_0 )}\\ 
= \varphi'_f (t_0 ) + \varphi'_g (t_0 )  = (\varphi_f \varphi_g)' (t_0 ).
\end{multline*}
Since this holds for a.e. $t_0$ with $\int_{0}^{t_0} |fg|\ ds>0$, this proves H\"{o}lder's inequality.

\section{Duality}

Let $p\colon [0,1]\to (1,\infty)$ be a measurable function. Let $\X_n \subset L^{p(\cdot)}$, $n\in \N$, be the images of the contractive projections 
$P_n \colon   L^{p(\cdot)} \to X_n$, $(P_n f)(t) = 1_{1+ 1/n \leq p(t) \leq n} (t) f(t)$. Then in fact
\[L_{0}^{p(\cdot)}= \overline{\bigcup_n X_n}\subset L^{p(\cdot)}.\]

This is seen as follows: we claim that for each $f\in X$ we have 
\[\|f-P_n f\|+\|P_n f \|\to \|f\|,\quad n\to\infty.\] For fixed initial value $a>0$ and $\varphi_{f-P_n f}(0)=\varphi_{P_n f}(0)=\varphi_{f}(0)=a$ the analogous statement follows easily, 
since $\varphi^{1-p(t)}\to 1$ as $p(t)\searrow 1$. By the absolute continuity of the solutions we obtain that $\varphi_{P_n f}(1) \to \varphi_{f}(1)$ as $n\to\infty$ for any given initial value $a>0$. Thus 
$\varphi_{f-P_n f ,a} (1)\to a$ as $n\to \infty$ for any initial value $a>0$. By a diagonal argument we find a sequence $(m_n )_n$ such that $\varphi_{f-P_{m_n} f, 1/n} (1)\to 0$ as $n\to \infty$.
Since the solutions are non-decreasing with respect to their initial values, we obtain that $\|f - P_n f\|\to 0$ as $n\to\infty$.\qed

Let us 
denote by $J$ the `duality map' $J\colon L^{p(\cdot)} \to L^0$ , $p(\cdot)>1$,
\begin{equation}\label{eq: Jxt}
J(x)[t]=\sign(x(t))|x(t)|^{\frac{p(t)}{p^* (t)}},\quad p,p^* \in L^0 ,\ \frac{1}{p(t)}+\frac{1}{p^* (t)}=1.
\end{equation}

\begin{theorem}
If $1< \essinf_t p(t) \leq \esssup_t p(t) < \infty$ then for each $F\in (L^{p(\cdot)})^*$ there is $f\in L^{p^* (\cdot)}$ such that 
\[\langle F , x\rangle = \int x(t) f(t)\ dm(t),\quad \mathrm{for\ all}\  x\in L^{p(\cdot)} \] 
and the above duality induces an isometric isomorphism $(L^{p(\cdot)})^* \to L^{p^* (\cdot)}$.
Moreover, $(L_{0}^{p(\cdot)})^*$ is isometric to $\widetilde{L}^{p^* (\cdot)}$ with the above duality for a general $p\colon [0,1]\to (1,\infty)$.
\end{theorem}
\begin{proof}

It follows from an easy adaptation of H\"older's inequality that $\widetilde{L}^{p^* (\cdot)}\subset  (L^{p(\cdot)})^*$ in the sense that 
\[|F(x)|=\left|\int x\ f\ dm \right| \leq \|x\|_{p(\cdot)} \|f\|_{p^* (\cdot)},\quad x\in L^{p(\cdot)}\]
whenever $f \in \widetilde{L}^{p^* (\cdot)}$ is regarded as a function and $F$ is in the subspace with the usual identification \eqref{eq: Fx}. 

Let us begin by verifying the statement in the reflexive case, i.e. $\essinf_t p(t)>1$ and $\esssup_t p(t)<\infty$ 
(see Theorem \ref{thm: char}), so that we are actually studying a space $\X_n$ for a given $n$. 
Let $F\in (L^{p(\cdot)})^*$. By modifying the standard proof (see e.g. \cite[Prop. 2.17]{FA_book}) of the statement in the usual constant exponent case we obtain that there is $f$ such that 
\begin{equation}\label{eq: Fx}
\langle F, x\rangle = \int x(t) f(t)\ dm(t)
\end{equation}
holds for every $x \in L^\infty$.
 
\newcommand{\supp}{\mathrm{supp}}

Note that $\|F\|_{\X^*} \leq \|f\|_{p^* (\cdot)}$. By applying the continuity of $F$ on one hand, and Lebesgue's monotone convergence theorem on the other hand, we may approximate 
\[F(x_n ) \to F(x),\quad \int x_n \ f\ dm\to \int x\ f\ dm\]
by bounded functions $x_n \to x \in L^{p (\cdot)}$. Thus \eqref{eq: Fx} holds for all $x \in L^{p (\cdot)}$.

Next we check that $\|F\|_{\X^*} \geq \|f\|_{p^* (\cdot)}$ which yields that the dual $(L^{p(\cdot)})^*$ is (even isometrically) $L^{p^* (\cdot)}$. First we restrict our considerations to functions $x \in L^{p(\cdot)}$ which are essentially bounded. As in \cite{talponen3} we investigate standard form 
simple semi-norms $N$, 
\[|x|_{N} = |x|_{(\ldots (L^{p_1} (\mu_1 ) \oplus_{p_2} L^{p_2} (\mu_2 )) \oplus_{p_3} \ldots )
\oplus_{p_n} L^{p_n} (\mu_n )},\]
which approximate $p(\cdot)$ in the sense that
\[\tilde{p}_N \nearrow p(\cdot) \]
in measure. In this case we may assume that $1 < \essinf p(\cdot) \leq p_i$.
Note that these semi-norms correspond in a canonical way to Banach spaces 
\[(\ldots (L^{p_1} (\mu_1 ) \oplus_{p_2} L^{p_2} (\mu_2 )) \oplus_{p_3} \ldots )
\oplus_{p_n} L^{p_n} (\mu_n ) .\]
The duality of these spaces is understood, namely, it is easy to verify recursively that  
\begin{multline*}
((\ldots (L^{p_1} (\mu_1 ) \oplus_{p_2} L^{p_2} (\mu_2 )) \oplus_{p_3} \ldots )
\oplus_{p_n} L^{p_n} (\mu_n ))^* \\
= (\ldots (L^{p_{1}^*} (\mu_1 ) \oplus_{p_{2}^*} L^{p_{2}^*} (\mu_2 )) \oplus_{p_{3}^*} \ldots ) \oplus_{p_{n}^*} L^{p_{n}^*} (\mu_n ).
\end{multline*}
If $N$ denotes the semi-norm corresponding to the left hand space inside the parenthesis, then the semi-norm 
corresponding to the right hand space is denoted by $N^*$.  
Since the supports of $\mu_i$ are successive, we may consider these spaces as function spaces on the unit interval. Denote by $\supp(N) = \bigcup_i \supp (\mu_i )$ for the corresponding measures $\mu_i$ in the representation of the semi-norm in question. Note that 
\[\int_{\supp(N)} x J_{p_N} (1_{\supp(N)}x) dm= |x|_N \ |J_{p_N} (1_{\supp(N)}x) |_{N^*},\]
from the duality of the spaces. Indeed, here we apply the duality of $L^{p_i} (\mu_i )$ spaces and 
of direct products $\R \oplus_{p_i} \R$, together with the fact
\[J \left(\left(\int x^{p_i}\ d\mu_i \right)^{\frac{1}{p_i}}\right) = \left(\int (J x)^{p_{i}^*} \ d\mu_i \right)^{\frac{1}{p_{i}^*}}.\]

We have the following convergences in measure
\[J_{\tilde{p}_N  \to \tilde{p}_{N^*}} (1_{\supp(N)}x) \to J_{p} (x),\]
\[\frac{\partial}{\partial t} |1_{[0,t]\cap \supp(N)} x|_N \to \varphi_{x}' ,\quad \frac{\partial}{\partial t} |J_{p_N} (1_{[0,t]\cap \supp(N^)}x)|_{N^*} \to \varphi_{x^*}' \]
as $\tilde{p}_N \nearrow p(\cdot)$ in measure. It follows that
\begin{equation}\label{eq: Jx} 
\int x J_p (x) \ dm= \|x\|_{p(\cdot)} \ \|J_p (x)\|_{p^* (\cdot)}.
\end{equation}
In fact, by using the absolute continuity of the norm accumulation functions $\varphi_f$, we obtain by 
straightforward approximation argument that \eqref{eq: Jx} holds for all $x \in  L^{p(\cdot)}$ in the reflexive case. Thus $\|F\|_{\X^*} = \|f\|_{p^* (\cdot)}$.

Next we treat the non-reflexive case. As pointed above, it follows from H\"older's inequality that 
$\widetilde{L}^{p^* (\cdot)} \subset (L_{0}^{p(\cdot)})^*$ and $\|F\|_{(L_{0}^{p(\cdot)})^*}\leq \|f\|_{\widetilde{L}^{p^* (\cdot)}}$.
Pick $f\in \S_{\widetilde{L}^{p^* (\cdot)}}$. Denote  
\[X_n = \{1_{1+ 1/n \leq p(\cdot)\leq n}\ x\colon x\in L_{0}^{p(\cdot)}\},\quad n\in\N\]
and let $P_n$ be the corresponding band projections. 
Restrict $F\in (L_{0}^{p(\cdot)})^*$ corresponding to $f$ to the subspace $\bigcup_n X_n$.
This does not change the operator norm, since the subspace is dense.
It is easy to see that  $\|P_n f\|_{L^{p^* (\cdot)}} \to  \|f\|_{\widetilde{L}^{p^* (\cdot)}}$ as $n\to \infty$. 
Hence, by using the observations of the reflexive case, we may pick for each $\epsilon>0$ such $n$ and $x\in X_n$, $\|x\|_{L^{p(\cdot)}}=1$, that $|(P^{*}_n f)(x)|>1-\epsilon$. 
Thus we observe that $\widetilde{L}^{p^* (\cdot)} \subset (L_{0}^{p(\cdot)})^*$ is an isometric subspace. 

Finally, pick $F \in (L_{0}^{p(\cdot)})^*$. Restrict $F$ to $\bigcup_n X_n$. Since the projections $P_n$ commute, this produces a natural candidate for the representation, namely
$f=\lim_n P^{*}_n  F$, the limit taken point-wise a.e. Since each $P^*_n F \in L^{p^* (\cdot)}$ and $\|P^*_n F\|_{L^{p^* (\cdot)}} \leq \|F\|_{(L_{0}^{p(\cdot)})^*}$, we obtain that
$f \in \widetilde{L}^{p^* (\cdot)}$, although the above limit does \emph{not}, a priori, exist in the $\widetilde{L}^{p^* (\cdot)}$ norm.  In fact,
$\|f\|_{\widetilde{L}^{p^* (\cdot)}}= \|F\|_{(L_{0}^{p(\cdot)})^*}$ by the construction of the norms.
Let us verify that $f$ presents $F$. Pick $x\in L_{0}^{p(\cdot)}$.
Then 
\begin{multline*}
F(x)-\int x(t)\ f(t)\ dm(t)\\ 
= F(x- P_n x) + F(P_n x) - \left(\int (x-P_n x)\ f\ dm + \int P_n x\ f\ dm \right).
\end{multline*}
Here $F(x- P_n x)\to 0$ by the continuity of the functional and $\int (x-P_n x)\ f\ dm\to 0$ by H\"older's inequality. 
On the other hand, 
\[F(P_n x)= (P^{*}_n f) (x)=\int_{1+ 1/n \leq p(\cdot )\leq n} f(t)\ x(t)\ dm(t) =  \int P_n x\ f\ dm .\]
Thus $F(x)=\int x(t)\ f(t)\ dm(t)$ for all $x\in  L_{0}^{p(\cdot)}$. This concludes the proof.
\end{proof}

Given a function $g\colon [0,1]\to \R$ with finite variation, let us denote a special `variation norm' as follows:
\[\bigvee_{p(\cdot)^*} m_g = \bigvee_{p(\cdot)^*} g = \sup \left\{\int_{0}^1 f\ dm_g\colon f\in C[0,1],\ \|f\|_{p(\cdot)}\leq 1\right\}.\]
Here $m_g$ is the Lebesgue-Stieltjes measure induced by $g$.
For a continuously differentiable $g$ the notable special cases are 
\[\bigvee_{(p\equiv 1)^*} g = \mathrm{Lip}(g),\]
the best Lipschitz constant of $g$, and the usual total variation
\[\bigvee_{(p\equiv \infty)^*} g = \bigvee g.\]
The above notion is applied somewhat tautologically in the following result.
To allow for integrating non-continuous functions easily, we will integrate in the more general Lebesgue-Stieltjes sense in taking duality. Thus, let $m_g$ be the Lebesgue-Stieltjes measure induced by $g$.

\begin{theorem}
Let $p\colon [0,1]\to (1,\infty)$ be measurable such that $L^{p(\cdot)}$ is a Banach space and let
\[\X :=\overline{C[0,1]}\subset L^{p(\cdot)} .\]
Then the dual space $\X^*$ elements are Lebesgue-Stieltjes measures $m_g$ with finite $\bigvee_{p(\cdot)^*} m_g$ variation.
The dual space is endowed with the norm 
\[\|m_g \|_{\X^*} =  \bigvee_{p(\cdot)^*} m_g\] 
and the duality is given by
\[\langle F, x\rangle =\int_{0}^1 x(t)\ dm_g (t),\quad x\in \X,\]
the Lebesgue integral with Lebesgue-Stieltjes measure $m_g$, induced by $g(t)=F(1_{[0,t)})$ for $F\in \X^*$. 
\end{theorem}

\begin{proof}
Let us begin by studying continuous linear functionals $F$ on the normed space $C[0,1]\subset L^{p(\cdot)}$. 
Since $\|\cdot \|_{p(\cdot)} \leq e\|f\|_\infty$ (see \cite{talponen3}), we obtain that each $F \in$ \mbox{$(C[0,1],\|\cdot\|_{p(\cdot)} )^*$} is also bounded
with respect to the norm $\|\cdot \|_{\infty}^*$. Thus $F \in (C[0,1],\|\cdot\|_{p(\cdot)} )^* \subset (C[0,1],\|\cdot\|_{\infty} )^*$ with the usual duality
\[\langle F, f\rangle =  \int f(t)\ dg(t),\quad g(t)=F(1_{[0,t)})\]
and
\[\bigvee g \leq  e \|F\|_{(C[0,1],\|\cdot\|_{L^{p(\cdot)}})^* }    .\]

We note that $F$ is a continuous linear functional on $(C[0,1],\|\cdot\|_{L^{p(\cdot)}})$, the above duality holds, if and only if 
\[\langle F , f\rangle = \int f\ dm_g ,\quad f\in C[0,1],\] 
$g(t)=F(1_{[0,t)})$. Here $\|F\|_{\X^*}=\bigvee_{p(\cdot)^*} m_g$ by the definition of the special variation.
 
Let us verify that the above integral representation extends continuously to the closure $\overline{C[0,1]}\subset L^{p(\cdot)}$ for each $F\in \X^*$. 
Fix $x\in \overline{C[0,1]}\subset L^{p(\cdot)}$. Pick $(x_n ) \subset C[0,1]$ such that $\|x_n - x\|_{L^{p(\cdot)}}\to 0$ as $n\to \infty$. Since $(x_n)$ is Cauchy, we can extract a subsequence $(n_j)$ 
such that $x_{n_1}+\sum_j x_{n_{j+1}}-x_{n_j}=x$ unconditionally in the $L^{p(\cdot)}$-norm and $\sum_j \|x_{n_{j+1}}-x_{n_j}\|_{p(\cdot)} <\infty$. It follows from the definition of $\bigvee_{p(\cdot)^*} m_g$ that then 
\begin{equation}\label{eq: dm_g}
\sum_j \left|\int (x_{n_{j+1}}-x_{n_j})(t)\ dm_g (t)\right| <\infty . 
\end{equation}
By passing to a further subsequence and modifying all the functions $x_{n_j}$ and $x$ in a $m_g$-null set we may assume that $x_{n_1}(t)+\sum_j (x_{n_{j+1}}-x_{n_j})(t)=x(t)$ for every $t$.

Consider the Banach space $L^1 (m_g )$. We obtain from \eqref{eq: dm_g} that $x_{n_j} \to y$ in the norm $\|\cdot\|_{L^1 (m_g )}$. Also, we observe that $y(t)=x(t)$ for $m_g$-a.e. $t$ by convergence in $m_g$-measure 
considerations. We conclude that 
\[\int x_{n_j} (t)\ dm_g (t) \to \int x (t)\ dm_g (t),\quad j\to\infty.\]
It is easy to see that the above convergence does not depend on the particular selection of the approximating Cauchy sequence of continuous functions. 
\end{proof}

\section{$L^{p(\cdot)} (\mu)$ spaces }

Let us consider an equivalent measure $\mu \sim m$ on the unit interval and $\frac{d\mu}{dm}$ with 
\[\mu(A)=\int_A \frac{d\mu}{dm}(t)\ dm(t)\]
for all Borel sets $A$. The above Radon-Nikodym derivative need not be integrable. Going back to the heuristical derivation of the norm-determining
ODE and repeating the considerations with $L^p (\mu)$ in place of $L^p$ under the assumption that $\frac{d\mu}{dm}(t)$ is a continuous function, 
we arrive at the following ODE:
\begin{equation}\label{eq: mudef}
\varphi(0 )=0^+ ,\ \varphi'(t)=\frac{d\mu}{dm}(t)\frac{|f(t)|^{p(t)}}{p(t)}\varphi(t)^{1-p(t)}\quad \mathrm{for}\ m\mathrm{-a.e.}\ t\in [0,1].
\end{equation}
Similarly as above we define a class of functions together with a norm (for a general $\mu \sim m$) and we denote this space by $L^{p(\cdot)}(\mu)$. This can be regarded as a `weighted $L^{p(\cdot)}$ space'. 
Recall that $L^p ([0,1])$ and $L^p (\R)$ are isometric; the same reasoning extends to our setting.

\begin{proposition}\label{prop: m_mu}
Let $p\colon [0,1]\to [1,\infty)$ be measurable such that $L^{p(\cdot)}$ is a Banach space and 
$\mu\sim m$ . Then $L^{p(\cdot)}(\mu)$ is a Banach space as well and the mapping
\[T\colon f(t)\mapsto \left(\frac{d\mu}{dm}(t)\right)^{-\frac{1}{p(t)}} f(t)\]
is a surjective linear isometry $L^{p(\cdot)}\to L^{p(\cdot)}(\mu)$.
\end{proposition}
\begin{proof}
Clearly the mapping is linear. Isometry follows by calculation:
\begin{multline*}
\varphi_{\mu,T(f)}'(t)=\frac{d\mu}{dm}(t)\frac{\left|\left(\frac{d\mu}{dm}(t)\right)^{-\frac{1}{p(t)}} f(t)\right|^{p(t)}}{p(t)}\varphi_{\mu,T(f)}(t)^{1-p(t)}\\
=\frac{|f(t)|^{p(t)}}{p(t)}\varphi_{\mu,T(f)}(t)^{1-p(t)}
=\frac{|f(t)|^{p(t)}}{p(t)}\varphi_{m,f}(t)^{1-p(t)}=\varphi_{m,f}'(t).
\end{multline*}
Indeed, a moment's reflection involving a joint positive initial value justifies the fact $\varphi_{\mu,T(f)}=\varphi_{m,f}$. Surjectivity follows by observing that 
\[f (t) \mapsto \left(\frac{d\mu}{dm}(t)\right)^{\frac{1}{p(t)}} f(t)\]
defines the inverse of the operator. Thus the class $L^{p(\cdot)}(\mu)$ is a Banach space as an (isometrically) isomorphic copy of $L^{p(\cdot)}$.
\end{proof}

\subsection{Applications of changing density}
\begin{theorem}\label{thm: char}
Let $p\colon [0,1] \to (1,\infty)$ be measurable such that $L^{p(\cdot)}$ is a Banach space. The following conditions are equivalent:
\begin{enumerate}
\item[(1)]{$L^ {p(\cdot)}$ is uniformly convex and uniformly smooth.}
\item[(2)]{$L^ {p(\cdot)}$ is reflexive.}
\item[(3)]{$L_{0}^ {p(\cdot)}$ contains neither $\ell^1$, nor $c_0$ almost isometrically.}
\item[(4)]{$\essinf_t p(t)>1$ and $\esssup_t p(t)<\infty$.}
\end{enumerate}
\end{theorem}
\begin{proof}
The implications $(1)\implies (2)\implies (3)$ are clear. 

The direction $(3) \implies (4)$. Suppose that $\essinf_t p(t)=1$. We will show that then $L_{0}^ {p(\cdot)}$ contains an isomorphic copy of $\ell^1$ for any isomorphism constant $C>1$.

By the compactness of the unit interval we can find a point $t_0$ such that 
\[\essinf_t 1_{(t_0-\epsilon , t_0 +\epsilon)}(t) p(t)=1\quad \mathrm{for\ each}\ \epsilon >0.\]
Indeed, assume that this is not the case and consider a suitable open cover of open intervals 
$(t_0-\epsilon , t_0 +\epsilon)$, so that there is a finite subcover contradicting $\essinf_t p(t)=1$. Therefore we may extract a sequence $(A_n)$ of measurable subsets of the unit interval with positive measure such that the following conditions hold:
\begin{enumerate}
\item{$\sup p|_{A_n}\searrow 1$ as $n\to \infty$.}
\item{Either $\max A_n < \min A_{n+1}$ for all $n$ or  $\max A_n > \min A_{n+1}$ for all $n$.}
\end{enumerate}

Fix a rapidly decreasing sequence of exponents $p_i \searrow 1$ such that 
\begin{equation}\label{eq: prod}
\prod_i \|I\colon \ell^{p_i}(2) \to \ell^{1}(2)\|<1+\epsilon.
\end{equation}
     
We can find a strictly increasing sequence $(n_i)$ such that $p_i \geq p|_{A_{n_i}}$ for each $i\in \N$. 

Let $\mu$ be an equivalent measure on the unit interval such that $\mu(A_{n_i})=1$ for $i\in \N$. In proving the claim it suffices study $L^{p(\cdot)}(\mu)$ in place of  $L^{p(\cdot)}$, since these spaces are isometric.
Put $\tilde{p}(t)=\max(1,\sum_i p_i 1_{A_{n_i}}(t))$. 

Define a mapping $T \colon \ell^1 \to L^{p(\cdot)}(\mu)$ by putting 
\[T((x_i ))=\sum_{i} x_i  1_{A_{n_i}}\]
where the sum is defined point-wise a.e.

We follow the arguments in \cite{talponen} involving sequence space semi-norms arising as follows.
For $(x_n) \in \ell^0$ we put
\[(\ldots (|x_1 | \boxplus_{p_1} |x_2 |) \boxplus_{p_2} |x_3 |) \boxplus_{p_3} \ldots \boxplus_{p_{n-1}}|x_{n}|)
\boxplus_{p_n} |x_{n+1}|,\]
in case $(A_n)$ is increasing, or the analogous left-handed version if $(A_n)$ is decreasing: 
\[|x_1 | \boxplus_{p_1} (|x_2 | \boxplus_{p_2} (|x_3 |\boxplus_{p_3} \ldots \boxplus_{p_{n-2}}(|x_{p_{n-1}}|\boxplus_{p_{n-1}}(|x_n | \boxplus_{p_n} |x_{n+1}|)\ldots ),\]
we observe that one may control inductively the difference of norms when one changes the values of the exponents $p_i$ by using \eqref{eq: prod}.   
That is, 
\[\left\|\sum_i x_i 1_{A_{n_i}}\right\|_{L^{p(\cdot)}(\mu)} \geq \frac{1}{1+\epsilon} \sum_i \|x_i 1_{A_{n_i}}\|_{L^{p(\cdot)}(\mu)}.\]
Thus, $\|T^{-1} \colon T(\ell^1 ) \to \ell^1\| \leq 1+\epsilon$.

Similarly, by passing to subsequences of $(A_n)$ multiple times we obtain that
\begin{multline*}
\sum_i \|x_i 1_{A_{n_i}}\|_{L^1 (\mu)}= \left\|\sum_i x_i 1_{A_{n_i}}\right\|_{L^1 (\mu)}  \leq (1+\epsilon) \left\|\sum_i x_i  1_{A_{n_i}}\right\|_{L^{p(\cdot)}(\mu)}\\
\leq (1+2\epsilon) \left\|\sum_i x_i 1_{A_{n_i}}\right\|_{L^{\tilde{p}(\cdot)}(\mu)} \leq (1+3\epsilon) \|(x_n)\|_{\ell^{p_\cdot }}\\
\leq (1+4\epsilon)\|(x_n)\|_{\ell^1}= (1+4\epsilon) \sum_i  \|x_i 1_{A_{n_i}}\|_{L^{1}(\mu)}.
\end{multline*}
Indeed, analyzing the $L^{p(\cdot)}$-differential equation shows that for a constant function the values of the derivative uniformly approximate $|f(t)|$ as $p(t)\searrow 1$.
Thus $\|T\|\leq 1+\epsilon$. This shows that the space contains $\ell^1$ almost isometrically.

Next, assume that $\esssup_t p(t)=\infty$. We will show that $L_{0}^ {p(\cdot)}$ contains $c_0$ almost isometrically. 
We may again without loss of generality make some assumptions about the equivalent measure, namely, that $\mu([0,1])=1$ and
\[\mu(\{t\in [0,1]\colon p(t)>r\})^{\frac{1}{r}} \to 1,\quad r\to \infty.\]
We will partition each set $\{t\in [0,1]\colon n<p(t)\leq n+1\}$
to measurable subsets of equal $\mu$-measure, call them $A^{(1)}_{n,0}$ and $A^{(1)}_{n,1}$. (Possibly both the subsets have measure $0$.)
Divide $A^{(1)}_{n,1}$ again to two subsets of equal measure, $A^{(2)}_{n,0}$ and $A^{(2)}_{n,1}$.
We proceed recursively in this manner to construct sets $A^{(k)}_{n,\theta}$, $k,n\in \N$, $\theta\in \{0,1\}$. Let $A^{(k)}_{j}=\bigcup_{n\geq j} A^{(k)}_{n,0}$. 
Observe that 
\[\mu(A^{(k)}_{j})=2^{-k}\mu(\{t\in [0,1]\colon p(t)>j\}),\quad k,j\in \N.\]
Note that 
\begin{equation}\label{eq: muA}
\lim_{j\to \infty} \mu(A^{(k)}_{j})^{\frac{1}{j}}=\lim_{j\to \infty}(2^{-k})^{\frac{1}{j}}\mu(\{t\in [0,1]\colon p(t)>j\})^{\frac{1}{j}}=1,\quad k\in\N.
\end{equation}

Assume first that $1_{A^{(n)}_{j}} \in L^{p(\cdot)}(\mu)$, although this is not necessarily the case. Define an operator  $T\colon c_{00} \to L^{p(\cdot)}(\mu)$ by 
\[T((x_n ))= \sum_n x_n 1_{A^{(n)}_{j}}\]
defined point-wise a.e. Clearly $\|T\|\leq \|1\|_{\widetilde{L}^{p(\cdot)}(\mu)}$. In fact, by choosing a large enough $j$ we get that $\|T\|\leq 1+ \epsilon$.
Indeed, observe that if $\varphi(t) \geq 1$ then $\frac{1}{j}\varphi^{1-j} (t)$ becomes small for a large $j$. Thus
\[(1+\eps) \max_n |x_n | \geq \|T((x_n))\|_{L^{p(\cdot)}(\mu)}\geq \max_n \|T(x_n e_n)\|_{L^{p(\cdot)}(\mu)}.\] 
Here $(e_n)$ is the canonical vector basis of $c_{00}$ and $(T(e_n))_n \subset  L^{p(\cdot)}(\mu)$ is a $1$-unconditional sequence. To show the claim it is required to check that 
\[\|T(e_n)\|_{L^{p(\cdot)}(\mu)}\geq 1-\eps,\quad n\in\N.\]
This is seen as follows, first observe that 
\[ \|1_{A^{(n)}_0 }\|_{L^{p(\cdot)}(\mu)}\geq \|1_{A^{(n)}_j}\|_{L^{p(\cdot)}(\mu)} .\] 
Then observe that for each $\eps>0$ there is $j \in \N$ such that 
\[\frac{1}{p(\cdot)}(\varphi (t) )^{1-p(\cdot)}\geq \frac{1}{j}(\varphi (t) +\eps)^{1-j},\quad p(\cdot)\geq j,\ \varphi(t)+\eps \leq 1.\]
This reads 
\begin{equation}\label{eq: pj}
\|1_{A_{j}^n}\|_{L^{p(\cdot)}(\mu)} \geq \|1_{A_{j}^n}\|_{L^j (\mu)} - \eps
\end{equation}
and further
\begin{equation}\label{eq: limsup}
\|1_{A_{0}^n}\|_{L^{p(\cdot)}(\mu)} \geq \limsup_{j\to\infty} \|1_{A_{j}^n}\|_{L^j (\mu)}.
\end{equation}

Recall that
\begin{equation}\label{eq:  to1}
\|1_{A^{(n)}_j}\|_{L^{j}(\mu)}=(2^{-n})^{\frac{1}{j}}\mu(\{t\in [0,1]\colon p(t)>j\})^{\frac{1}{j}}\to 1,\quad j\to \infty.
\end{equation}

We made an additional assumption during the course of the proof that $1_{A^{(n)}_{j}}$ is included in the space. This assumption can be removed by observing that we may restrict the support of these functions 
to suitable sets $\{t\colon p(t) \leq p^{(n)}\}$, so that the positive-initial-value solutions become Lipschitz with a large constant and 
such that simultaneously \eqref{eq: limsup} and \eqref{eq:  to1} hold up to an extra $\eps$.
Thus $L_{0}^ {p(\cdot)}$ contains $c_0$ almost isometrically.

The direction $(4)\implies (1)$. Here we will follow the analogous argument in the setting of $\ell^{p(\cdot)}$ spaces.
We will require the notions of upper $p$-estimate and lower $q$-estimate of Banach lattices.
If $\X$ is a Banach lattice and $1\leq p\leq q<\infty$ then the upper $p$-estimate and the lower $q$-estimate, respectively, 
are defined as follows:\\
\begin{equation*}
\begin{array}{l}
\left|\left|\sum_{1\leq i\leq n} x_{i}\right|\right|\leq \bigboxplus_{1\leq i\leq n}^p \|x_{i}\| , \phantom{\bigg |}\\
\left|\left|\sum_{1\leq i\leq n} x_{i}\right|\right|\geq \bigboxplus_{1\leq i\leq n}^q \|x_{i}\| ,\phantom{\bigg |}
\end{array}
\end{equation*}
respectively, for any vectors $x_{1},\ldots, x_{n}\in \X$ with pairwise disjoint supports. These estimates involve multiplicative coefficients which are 
taken to be $1$ in this treatment. We will apply the fact that a Banach lattice, which satisfies an upper $p$-estimate and a lower $q$-estimate 
for some $1<p<q<\infty$ with constants $1$ is both uniformly convex and uniformly smooth (with the respective power types), see \cite[1.f.1, 1.f.7]{LT}. 

Let $1<p=\essinf_t p(t)$ and $\esssup_t p(t)=q<\infty$. We claim that $L^{p(\cdot)}$ satisfies the respective estimates for these $p$ and $q$.
To check the upper $p$-estimate, let $f_k$, $1\leq k \leq n$, be disjointly supported functions in $L^{p(\cdot)}$. 
Observe that if $\X$ and $\Y$ satisfy the upper $p$-estimate, then $\X \oplus_r \Y$ satisfies it as well for $r\geq p$. Indeed, 
\[\bigboxplus_{i}^p \|(x_i , y_i)\|_{\X \oplus_{r} \Y}\geq \bigboxplus_{i}^p \|x_i\|_\X \ \boxplus_r \   \bigboxplus_{i}^p\|y_i \|_{\Y}
\geq \left\|\sum_i x_i\right\|_\X  \boxplus_r \left\|\sum_i y_i\right\|_\Y = \left\|\sum_i (x_i, y_i )\right\|_{\X \oplus_r \Y}\]
where we applied the direct sum norm twice, Proposition \ref{prop: boxplus} and the upper $p$-estimate of $\X$ and $\Y$.  
Thus, using this observation inductively on the semi-norms $\mathcal{N}$ we obtain the statement by approximation.

Alternative route. By a simple argument using the definition of outer measure we see that each simple semi-norm can be approximated point-wise from below with other semi-norms 
of the type $\|\cdot\|_{(\ldots (L^{p_1}(\mu_1 ) \oplus_{r_2} L^{p_2}(\mu_2 )) \oplus_{r_3}\ldots \oplus_{r_m} L^{p_{m}}(\mu_m ))}$, $\essinf_t p(t)\leq r_i \leq \esssup_t p(t)$, such that only one of the functions $f_k$ is supported on 
the support of a given $\mu_i $. We may interpret the values of the semi-norms as norms of finite $\ell^{p(\cdot)}$ sequences 
\[f\mapsto (|f|_{L^{p_1}(\mu_1 )},|f|_{L^{p_2}(\mu_2 )},\ldots , |f|_{L^{p_{m}}(\mu_m )})\]
and then the supports of the sequences are disjoint for disjointly supported functions $f_k$. We apply the fact proved in \cite{talponen} which states that for disjointly supported $\ell^{p(\cdot)}$ sequences 
we have the upper $p$-estimate for $p=\inf_t p_t$.  From these considerations it follows that  also disjointly supported $L^{p(\cdot)}$ functions satisfy the upper $p$-estimate for $p=\essinf_t p(t)$.
 
The argument for lower $q$-estimates is analogous. This concludes the proof.
\end{proof}

Next, our aim is to build a kind of \emph{universal} $L^{p(\cdot)}$ space. We will study a modification of Topologist's Sine Curve as follows:
\[p_0 (t)=\frac{1}{1-t}\sin\left(\frac{1}{1-t}\right)+\frac{1}{1-t}+1,\quad 0\leq t <1 .\]

\begin{theorem}
Let $p_0$ be as above. Suppose that $p$ is any exponent such that $L^{p(\cdot)}$ is a Banach space. Then 
$L^{p(\cdot)}$ is finitely representable in $L^{p_0 (\cdot)}$. Assume further that $p\colon [0,1)\to [1,\infty)$ is a $C^1$-function, not constant on any proper interval and that $p'$ changes its sign finitely many times on each interval $[0,a]\subset [0,1)$. Then there is an isometric linear embedding $L^{p(\cdot)}\to L^{p_0 (\cdot)}$ onto a projection band.
\end{theorem}

\begin{proof}
We omit the argument for the first part of the statement.
Towards the second part, according to the assumptions we find a sequence of open subintervals $\Delta_{n} \subset [0,1]$, $n\in\N$, with $\sup \Delta_n = \inf \Delta_{n+1}$ such that the sign of $p'$ does not properly change on the intervals $\Delta_{n}$.
Moreover, we may assume that $|p' (x)|>0$ for $x \in \bigcup_n \Delta_n$. We may choose this collection to be almost a cover in the sense that $m\left([0,1] \setminus \bigcup_n \Delta_{n}\right)=0$.

Now, $p$ is monotone on each $\Delta_{n}$. By the construction of $p_0$ we can find a sequence of open intervals $\Delta_{n}' \subset [0,1]$, $n\in\N$, with $\sup \Delta_{n}' \leq \inf \Delta_{n+1}'$ such that 
there is a $C^1$-diffeomorphism $T_n \colon \Delta_{n} \to \Delta_{n}'$ with $p|_{\Delta_{n} }= p_0 \circ T_n$.  

By taking the union of the graphs of $T_{n}$, i.e. by 'gluing together' these mappings, we define a mapping $T$ defined a.e. on $[0,1]$, which has the property that $p(x)=p_ 0 (T(x))$ for a.e. $x\in [0,1]$. 

Let us define absolutely continuous measures $\nu$ and $\mu$ on $[0,1]$ given by $\frac{d\nu}{dm}(t)=|p'(t)|$ 
and $\frac{d \mu}{d m}(t)= |p_{0}' (t)|$ for $m$-a.e. $t$.

By making suitable identifications via $T$ we may consider $L^{p(\cdot)}(\nu)$ as a subspace of $L^{p_0 (\cdot)}(\mu)$.
Both $\nu$ and $\mu$ can be thought as variation measures corresponding to $p$ and $p_0$, respectively.
Thus it is easy to see that $T$ is a $\nu$-$\mu$-measure-preserving mapping and
\[\|1_{T([0,1])}f\|_{L^{p_0 (\cdot)}(\mu)}=\|f \circ T\|_{L^{p(\cdot)}(\nu)}=\|g\|_{L^{p(\cdot)}(\nu)}\] 
for $f\in L^{p_0 (\cdot)}(\mu)$ such that $f \circ T = g\in L^{p(\cdot)}(\nu)$. Indeed, by using the absolute continuity of the solutions we observe that values of $f$ outside $T([0,1])$ do not influence the norm.

This way we may apply Proposition \ref{prop: m_mu} to observe that $G\colon L^{p(\cdot)}(\nu) \to L^{p_0 (\cdot)}(\mu)$ given by 
\[G(f)[t]=\left(\frac{d\nu}{dm}(T^{-1} (t))\bigg/ \frac{d\mu}{dm}(t)\right)^{\frac{1}{p_0 (t)}}f(T^{-1} (t)) \quad \mathrm{if}\ t\in T([0,1]),\] 
and $G(f)[t]=0$ otherwise, defines the required isometry. Note that in integrating with a change of variable 
above the map $T^{-1}$ is $\mu$-$\nu$-measure-preserving.
\end{proof}

\subsection*{Acknowledgments}

This work has received financial support from the V\"ais\"al\"a foundation, the Finnish Cultural Foundation and the Academy of Finland Project \# 268009.


\begin{thebibliography}{DGZ}
\bibitem[BO31]{bo}
Z. Birnbaum; W. Orlicz,  \"Uber die Verallgemeinerung des Begriffes der zueinander Konjugierten Potenzen, 
Studia Math. 3 (1931), 1--67.
\bibitem[ACK98]{kalton1} G. Androulakis, C. Cazacu, N. Kalton, Twisted sums, Fenchel-Orlicz spaces and property (M). Houston J. Math. 24 (1998), 105--126.
\bibitem[CL55]{DE_book}  E. Coddington, N. Levinson, Theory of ordinary differential equations. McGraw-Hill Book Company, Inc., New York-Toronto-London, 1955.
\bibitem[DR03]{die0}
L. Diening, M. Ruzicka, Calderon-Zygmund operators on generalized Lebesgue spaces $L^{p(\cdot)}$ and problems related to fluid dynamics. J. Reine Angew. Math. 563 (2003), 197--220.
\bibitem[Fabian et al. 01]{FA_book} M. Fabian, P.  Habala, P. H\'ajek, V. Montesinos Santaluci­a, J. Pelant, V. Zizler, Functional analysis and infinite-dimensional geometry. CMS Books in Mathematics/Ouvrages de Mathématiques de la SMC, 8. Springer-Verlag, New York, 2001.
\bibitem[JKL93]{JKL}
J.E. Jamison, A. Kami\'nska, P.-K. Lin, Isometries of Musielak-Orlicz spaces II, Studia Math. 104 (1993), 75--89.
\bibitem[KR91]{kovacik} O. Kov\'a\u{c}ik, J. R\'akosn\'ik, On spaces $L^{p(x)}$ and $W^{k,p(x)}$. Czechoslovak Math. J. 41 (1991), 592--618.
\bibitem[LT96]{LT} J. Lindenstrauss, L. Tzafriri, Classical Banach Spaces (Classics in Mathematics) Paperback, Springer 1996.
\bibitem[Lux55]{lux}
W. Luxemburg, Banach function spaces, T.U. Delft (1955) (Thesis).
\bibitem[Mal11]{Maligranda}
L. Maligranda, Hidegoro Nakano (1909-1974) - on the centenary of his birth.
In: Banach and Function Spaces III, Proceedings of the Third International Symposium on Banach and Function Spaces 2009, Sep. 14-17, 2009, Kitakyushu-Japan (Editors M. Kato, L. Maligranda and T. Suzuki), Yokohama Publishers 2011.
\bibitem[Mar39]{mar}
J. Marcinkiewicz, Sur l'interpolation d'op\'erations C.R. Acad. Sci. Paris, 208 (1939) 1272--1273
\bibitem[Mus83]{mus}
J. Musielak, Orlicz spaces and modular spaces, Lecture Notes in Mathematics, Vol. 1034 Springer-Verlag, Berlin/New York (1983).
\bibitem[NS12]{nakai} E. Nakai, Y. Sawano, Hardy spaces with variable exponents and generalized Campanato spaces. J. Funct. Anal. 262 (2012), 3665--3748.
\bibitem[RR91]{rr} M. Rao, Z. Ren,
Theory of Orlicz spaces, Monographs and Textbooks in Pure and Applied Mathematics, vol. 146 Marcel Dekker, Inc., New York (1991).
\bibitem[RR02]{rr2} M. Rao, Z. Ren, Applications Of Orlicz Spaces, CRC Press, 2002.
\bibitem[Sob41]{Sobczyk}
A. Sobczyk, Projections in Minkowski and Banach spaces, Duke Math. J. 8 (1941), 78--106.
\bibitem[Tal11]{talponen} J. Talponen, A natural class of sequential Banach spaces. Bull. Pol. Acad. Sci. Math. 59 (2011), 185--196.
\bibitem[Tal15]{talponen2} Note on Order-Isomorphic Isometric Embeddings of Some Recent Function Spaces, Journal of Function Spaces (2015), 6 pp.\ http://dx.doi.org/10.1155/2015/186105
\bibitem[Tal17]{talponen3} ODE to Lp norms, Studia Math. (to appear), doi: 10.4064/sm8561-8-2016\\
arxiv.org/abs/1402.0528
\end{thebibliography}
\end{document}